\documentclass[preprint,12pt]{elsarticle}

\usepackage{amsthm}
\usepackage{amsmath,amssymb,txfonts}
\usepackage{color,bm,comment}
\usepackage[english]{babel}

\newtheorem{theorem}{Theorem}

\newtheorem{prop}{Proposition}
\newtheorem{cor}{Corollary}
\newtheorem{remark}{Remark}

\theoremstyle{definition}

\def\re{\mathbb{R}}
\def\N{\mathbb{N}}

\def\({\left(}
\def\){\right)}
\def\[{\left[}
\def\]{\right]}
\def\pd{\partial}
\def\lap{\Delta}

\def\ep{\varepsilon}
\def\w{\omega}
\def\la{\lambda}

\def\Sh{\mathbb{S}}

\newtheorem{ThmA}{Theorem A}

\newtheorem{ThmB}{Theorem B}

\newtheorem{ThmC}{Theorem C}

\begin{document}

\begin{frontmatter}



\title{On the compactness of the (non)radial Sobolev spaces
}


\author[SM]{Shuji Machihara}
\ead{machihara@rimath.saitama-u.ac.jp}

\author[MS]{Megumi Sano\corref{Sano}\fnref{label1}}
\ead{smegumi@hiroshima-u.ac.jp}
\fntext[label1]{Corresponding author.}

\address[SM]{Department of Mathematics, Saitama University, Saitama, 338-8570, Japan}
\address[MS]{Laboratory of Mathematics, School of Engineering,
Hiroshima University, Higashi-Hiroshima, 739-8527, Japan}


\date{\today}

\begin{abstract}
In this note, we give the affirmative answer of the question in \cite{S(NA)}, which is a compactness result of the non-radial Sobolev spaces.  
As an application, we show the existence of an extremal function of the critical Hardy inequality under spherical average zero. 
Next, we give an improvement of the compactness results of the radial Sobolev spaces in \cite{ES}.   
In Appendix, we give an alternative proof of Hardy type inequalities under spherical average zero. 
\end{abstract}

\end{frontmatter}



%
%
\section{Introduction}\label{S Intro}

Let $N \ge 3$ and $2^* = \frac{2N}{N-2}$.  
The Sobolev embedding $H^{1} (\re^N) \hookrightarrow L^q (\re^N)$ is non-compact for any $q \in [2, 2^*]$. However, it is well-known that if we restrict $H^1 (\re^N)$ to the radial Sobolev space $H^1_{\rm rad} (\re^N)$, then the embedding $H^{1}_{\rm rad} (\re^N) \hookrightarrow L^q (\re^N)$ becomes compact for any $q \in (2, 2^*)$ (Ref. \cite{St}). It is called {\it Strauss's radial compactness}. 
We can deny the possibility of non-compactness with respect to translation invariance thanks to the restriction of $H^1(\re^N)$ to $H^1_{\rm rad}(\re^N)$ such as the non-compact sequence $\{ u_m \}_{m=1}^\infty \subset H^1(\re^N)$ of the embedding $H^{1} (\re^N) \hookrightarrow L^q (\re^N)$, where $u_m (x) = u(x + x_m) \,(x \in \re^N), u$ is a smooth function on $\re^N$ and $|x_m| \to \infty \,(m \to \infty)$.
Therefore, we have the compactness and the non-compactness results of three embeddings as follows.

\begin{ThmA}(Strauss's radial compactness) 
Let $N \ge 3$ and $2 < q < 2^*$. Then
\begin{align*}
H^{1} (\re^N) \hookrightarrow L^q (\re^N)&: \,\text{non-compact},\\
H_{\rm rad}^{1} (\re^N) \hookrightarrow L^q (\re^N)&: \,\text{compact,}\\
\( H_{{\rm rad}}^1(\re^N) \)^{\perp} \hookrightarrow L^q (\re^N)&: \,\text{non-compact.}
\end{align*}
\end{ThmA}

\noindent
Here, $\( H_{{\rm rad}}^1(\re^N) \)^{\perp}$ is the orthogonal complement of the radial Sobolev space $H_{{\rm rad}}^{1} (\re^N)$ and the last one in Theorem A follows from Proposition \ref{Prop non-cpt} in Appendix. 
For simplicity, we call $\( H_{{\rm rad}}^1(\re^N) \)^{\perp}$ the non-radial Sobolev space.
Several generalizations of Strauss's radial compactness have been investigated, see \cite{Li,FZZ,HS}.
Note that even if we restrict $H^1 (\re^N)$ to $H^1_{\rm rad} (\re^N)$, the embedding $H^{1}_{\rm rad} (\re^N) \hookrightarrow L^q (\re^N)$ is still non-compact for $q=2$ or $2^*$. In fact, in the case where $q =2$, there is a non-compact sequence with respect to $L^2$-invariance, which is {\it a vanishing sequence} $\{ u_m \}_{m=1}^\infty \subset H^{1}_{\rm rad} (\re^N)$, where $u_m (x) = m^{-\frac{N}{2}} u\( \frac{x}{m} \) \,(x \in \re^N, m \in \N)$. On the other hand, in the case where $q =2^*$, there is a non-compact sequence with respect to $\dot{H}^1$ and $L^{2^*}$-invariance, which is {\it a concentration sequence} $\{ u_m \}_{m=1}^\infty \subset H^{1}_{\rm rad} (\re^N)$, where $u_m (x) = m^{\frac{N-2}{2}} u\( mx \) \,(x \in \re^N, m \in \N)$. As we can see from Strauss's radial compactness, we may deny the possibility of non-compactness 
under some restriction of $H^1(\re^N)$.
In fact, if we restrict $H^{1}_{\rm rad} (\re^N)$ to $H^{1}_{\rm rad} (\re^N) \cap L^p (\re^N)$, then the embedding $H^{1}_{\rm rad} (\re^N) \cap L^p (\re^N ) \hookrightarrow L^2 (\re^N)$ becomes compact for $p \in [1,2)$ and the embedding $H^{1}_{\rm rad} (\re^N) \cap L^p (\re^N ) \hookrightarrow L^{2^*} (\re^N)$ becomes compact for $p> 2^*$, see \cite{ES} Corollary 1. In \S \ref{S radial}, we extend these results to Lorentz spaces $L^{p,q}(\re^N)$.


In this note, we give another example of the compactness of the Sobolev embedding by answering the following question in \cite{S(NA)}: 

Let $B_1^N \subset \re^N$ be the unit ball, $a>1, q \ge 2$. Then, {\it is the embedding $\( H_{0, {\rm rad}}^1(B_1^2) \)^{\perp} \hookrightarrow L^q \( B_1^2 ; |x|^{-2} (\log \frac{a}{|x|})^{-1-\frac{q}{2}}\,dx \)$ compact?}

\noindent
This question comes from a heuristic consideration and calculation via harmonic transplantation which is proposed by Hersch \cite{H}. For a summary of harmonic transplantation, see e.g. \cite{ST(HT)} \S 3. 
First, we give the affirmative answer of the above question as follows. 

\begin{theorem}\label{Thm cpt}
Let $B_1^2 \subset \re^2$ be the unit ball, $a>1$ and $q \ge 2$. Then
 the embedding 
$\( H_{0, {\rm rad}}^1(B_1^2) \)^{\perp} \hookrightarrow L^q \( B_1^2 ; |x|^{-2} \( \log \frac{a}{|x|} \)^{-1-\frac{q}{2}}\,dx \)$ is compact. 
\end{theorem}

By combining Theorem \ref{Thm cpt} and the non-compactness of the embedding $H_{0, {\rm rad}}^{1} (B_1^2) \hookrightarrow L^q \( B_1^2 ; |x|^{-2} \( \log \frac{a}{|x|} \)^{-1-\frac{q}{2}}\,dx \)$, we have the following.
\begin{cor}\label{Cor cpt}(``Non-radial compactness'') 
Let $B_1^2 \subset \re^2$ be the unit ball, $a>1$ and $q \ge 2$. Then
\begin{align*}
H_0^{1} (B_1^2) \hookrightarrow L^q \( B_1^2 ; |x|^{-2} \( \log \frac{a}{|x|} \)^{-1-\frac{q}{2}}\,dx \)&: \,\text{non-compact,}\\
H_{0, {\rm rad}}^{1} (B_1^2) \hookrightarrow L^q \( B_1^2 ; |x|^{-2} \( \log \frac{a}{|x|} \)^{-1-\frac{q}{2}}\,dx \)&: \,\text{non-compact,}\\
\( H_{0, {\rm rad}}^1(B_1^2) \)^{\perp} \hookrightarrow L^q \( B_1^2 ; |x|^{-2} \( \log \frac{a}{|x|} \)^{-1-\frac{q}{2}}\,dx \)&: \,\text{compact.}
\end{align*}
\end{cor}

\noindent
In Theorem A, the non-compact embedding becomes compact under the restriction of $H^1 (\re^N)$ to $H^1_{\rm rad}(\re^N)$, while the non-compact embedding becomes compact under the restriction of $H_0^{1} (B_1^2)$ to $\( H_{0, {\rm rad}}^1(B_1^2) \)^{\perp}$ in Corollary \ref{Cor cpt}. In this sense, 
Corollary \ref{Cor cpt} implies an opposite phenomenon to Strauss's radial compactness. 


Second, as an application of Theorem \ref{Thm cpt} with $q=2$, we can obtain an extremal of the critical Hardy inequality under spherical average zero:
\begin{align*}
A_a \int_{B_1^2} \frac{|u|^2}{|x|^2 \( \log \frac{a}{|x|} \)^2} \,dx \le \int_{B_1^2} |\nabla u|^2 \,dx
\end{align*}
for any $u \in H_0^1 (B_1^2)$ with $\int_{0}^{2\pi} u(r\theta) \,d\theta = 0\, ({}^{\forall}r \in [0,1])$. 
Namely, we can obtain a minimizer of the following minimization problem $A_a$
\begin{align*}
A_a  &:= \inf \left\{ \, \frac{\int_{B_1^2} |\nabla u|^2 \,dx }{\int_{B_1^2} \frac{|u|^2}{|x|^2 (\log \frac{a}{|x|})^2} \,dx } \,\,\middle| \,\, u \in H_0^1 (B_1^2) \setminus \{ 0\}, \, \int_0^{2\pi} u(r\theta) \,d \theta =0 \,({}^{\forall}r \in [0,1] \,) \, \right\}\\
&= \inf \left\{ \, \frac{\int_{B_1^2} |\nabla u|^2 \,dx }{\int_{B_1^2} \frac{|u|^2}{|x|^2 (\log \frac{a}{|x|})^2} \,dx } \,\,\middle| \,\, u \in \( H_{0, {\rm rad}}^1 (B_1^2) \)^{\perp} \setminus \{ 0\} \, \right\}\\
&\ge B_a := \inf \left\{ \, \frac{\int_{B_1^2} |\nabla u|^2 \,dx }{\int_{B_1^2} \frac{|u|^2}{|x|^2 (\log \frac{a}{|x|})^2} \,dx } \,\,\middle| \,\, u \in H_0^1 (B_1^2) \setminus \{ 0\} \, \right\} = \frac{1}{4}.
\end{align*}
For the second equality, see e.g. \cite{BMO}. It is well-known that for any $a \ge 1$, there is no minimizer of $B_a = \frac{1}{4}$, see e.g. \cite{AS,II}. 

For the Hardy inequality under spherical average zero, see e.g.  \cite{BL,EF,BMO} or Proposition \ref{Prop R^N} in Appendix.

\begin{cor}\label{Cor A_a q=2}
$A_a > \frac{1}{4} = B_a$ for $a>1$ and $A_1 = \frac{1}{4} = B_1$ . Moreover, there exists a minimizer of $A_a$ if and only if $a>1$.
\end{cor}

Also, in the case where $q>2$, we obtain the following corollary.

\begin{cor}\label{Cor A_a q>2}
Let $q>2$ and $a>1$. Set
\begin{align*}
D_a  &:= \inf \left\{ \, \frac{\int_{B_1^2} |\nabla u|^2 \,dx }{\( \int_{B_1^2} \frac{|u|^q}{|x|^2 (\log \frac{a}{|x|})^{\frac{q}{2}+1}} \,dx \)^{\frac{2}{q}} } \,\,\middle| \,\, u \in H_0^1 (B_1^2) \setminus \{ 0\}, \, \int_0^{2\pi} u(r\theta) \,d \theta =0 \,({}^{\forall}r \in [0,1] \,) \, \right\}.
\end{align*}
Then there exists a minimizer of $D_a \,(>0)$. 
\end{cor}

Since the proof of Corollary \ref{Cor A_a q>2} is the same as the proof of Corollary \ref{Cor A_a q=2}, we omit the proof of Corollary \ref{Cor A_a q>2}. 
Set
\begin{align*}
G_a &:= \inf \left\{ \, \frac{\int_{B_1^2} |\nabla u|^2 \,dx }{\( \int_{B_1^2} \frac{|u|^q}{|x|^2 (\log \frac{a}{|x|})^{\frac{q}{2}+1}} \,dx \)^{\frac{2}{q}} } \,\,\middle| \,\, u \in H_0^1 (B_1^2) \setminus \{ 0\} \, \right\},\\
G_{a, {\rm rad}} &:= \inf \left\{ \, \frac{\int_{B_1^2} |\nabla u|^2 \,dx }{\( \int_{B_1^2} \frac{|u|^q}{|x|^2 (\log \frac{a}{|x|})^{\frac{q}{2}+1}} \,dx \)^{\frac{2}{q}} } \,\,\middle| \,\, u \in H_{0, {\rm rad}}^1 (B_1^2) \setminus \{ 0\} \, \right\},\\
D_a  &= \inf \left\{ \, \frac{\int_{B_1^2} |\nabla u|^2 \,dx }{\( \int_{B_1^2} \frac{|u|^q}{|x|^2 (\log \frac{a}{|x|})^{\frac{q}{2}+1}} \,dx \)^{\frac{2}{q}} } \,\,\middle| \,\, u \in \( H_{0, {\rm rad}}^1 (B_1^2) \)^{\perp} \setminus \{ 0\} \, \right\}.
\end{align*}
It is known that $G_{a, {\rm rad}}$ is not attained for any $a \in (1, \infty)$, see \cite{HK}. The second author \cite{S(JDE)} showed that there exists $a_* >1$ such that $G_a < G_{a, {\rm rad}}$ and $G_a$ is attained for $a \in (1, a_*)$, and $G_a = G_{a, {\rm rad}}$ and $G_a$ is not attained for $a > a_*$. We can interpret the existence of a minimizer of $G_a$ for $a \in (1, a_*)$ intuitively as the embedding $\( H_{0, {\rm rad}}^1(B_1^2) \)^{\perp} \hookrightarrow L^q \( B_1^2 ; |x|^{-2} \( \log \frac{a}{|x|} \)^{-1-\frac{q}{2}}\,dx \)$ is dominant when $G_a < G_{a, {\rm rad}}$, and the existence of a minimizer of $G_a$ comes from the effect of the compactness of the embedding shown in  Theorem \ref{Thm cpt}.

\begin{remark}($a=1$)
In the case where $a=1$, we have the followings.
\begin{align*}
H_0^{1} (B_1^2) &\not\hookrightarrow L^q \( B_1^2 ; |x|^{-2} \( \log \frac{1}{|x|} \)^{-1-\frac{q}{2}}\,dx \)\\
H_{0, {\rm rad}}^{1} (B_1^2) &\hookrightarrow L^q \( B_1^2 ; |x|^{-2} \( \log \frac{1}{|x|} \)^{-1-\frac{q}{2}}\,dx \)\\
\( H_{0, {\rm rad}}^1(B_1^2) \)^{\perp} &\not\hookrightarrow L^q \( B_1^2 ; |x|^{-2} \( \log \frac{1}{|x|} \)^{-1-\frac{q}{2}}\,dx \)
\end{align*}
For the first and the second one, see \cite{HK,S(JDE)}. The third one follows from the first and the second one. Therefore, we have $D_1 = G_1 = 0$. 
\end{remark}

The minimization problems associated with the Rellich type inequalities under spherical average zero are studied by \cite{CM}. 
Also, Hardy type inequalities with another average zero condition, which comes from Neumann problem, are studied by \cite{CPR,ST}.   

In the next section, we show Theorem \ref{Thm cpt}, Corollary \ref{Cor cpt} and Corollary \ref{Cor A_a q=2}. 
In \S \ref{S radial}, we show that the embedding $H^{1}_{\rm rad} (\re^N) \cap L^{p,\infty} (\re^N ) \hookrightarrow L^2 (\re^N)$ is compact for $p \in [1,2)$ and the embedding $H^{1}_{\rm rad} (\re^N) \cap L^{p, \infty} (\re^N ) \hookrightarrow L^{2^*} (\re^N)$ is compact for $p > 2^*$. Since $L^p (\re^N) \subsetneq L^{p, \infty} (\re^N)$, it is an improvement of Corollary 1 in \cite{ES}. 
In Appendix, we give an alternative proof of Hardy type inequalities under spherical average zero. 
In \cite{BMO} p.13-14, the optimality of the constant of the Hardy inequality under spherical average zero was shown on the whole space by using Fourier analysis and Lemma 3.8 in \cite{Y}.
On the other hand, our proof is available not only on the whole space but also on the ball. Furthermore, our proof is self-contained.

\noindent
{\bf Notation.} \,
$|A|$ denotes the Lebesgue measure of a set $A \subset \re^N$ and $\w_{N-1}$ denotes an area of the unit sphere in $\re^N$. 
$H_0^1 (B_1^N)$ is the completion of $C_0^\infty (B_1^N)$ with respect to the norm $\| \nabla (\cdot ) \|_2$.
Throughout this note, if a radial function $u$ is written as $u(x) = \tilde{u}(|x|)$ by some function $\tilde{u} = \tilde{u}(r)$, we write $u(x)= u(|x|)$ with admitting some ambiguity. 
Also, we use $C$ or $C_i \,(i \in \N)$ as positive constants. If necessary and these constants depend on $\ep$, they will be denoted by $C(\ep)$.

%
%
\section{Compactness of the non-radial Sobolev spaces: Proofs of Theorem \ref{Thm cpt}, Corollary \ref{Cor cpt} and Corollary \ref{Cor A_a q=2}}\label{S Proof}

First, we show Theorem \ref{Thm cpt} and Corollary \ref{Cor cpt}.

\begin{proof}{\it (Theorem \ref{Thm cpt})}
First, we assume that $q=2$. 
Let $\{ u_m \}_{m=1}^\infty \subset \( H_{0, {\rm rad}}^1(B_1^2) \)^{\perp}$ be a bounded sequence. Since $\( H_{0, {\rm rad}}^1(B_1^2) \)^{\perp}$ is a closed subspace of the reflexive Banach space $H_{0}^1(B_1^2)$, $\( H_{0, {\rm rad}}^1(B_1^2) \)^{\perp}$ is also the reflexive Banach space. Therefore, there exist $u \in \( H_{0, {\rm rad}}^1(B_1^2) \)^{\perp}$ and a subsequence $\{ u_{m_j} \}_{j=1}^\infty$ (we use  $\{ u_m \}$ again) such that $u_m \rightharpoonup u$ in $\( H_{0, {\rm rad}}^1(B_1^2) \)^{\perp}$ as $m \to \infty$. Since $H_0^1(B_1^2) \hookrightarrow L^2 (B_1^2)$ is compact, we have $u_m \to u$ in $L^2(B_1^2)$. Therefore, we have
\begin{align*}
\int_{B_1^2} \frac{|u_m - u|^2}{|x|^2 (\log \frac{a}{|x|})^2}\,dx 
&\le \( \log \frac{a}{\ep} \)^{-2} \int_{B^2_\ep} \frac{|u_m - u|^2}{|x|^2 }\,dx
+ C(\ep) \int_{B^2_1 \setminus B^2_\ep} |u_m -u|^2 \,dx \\
&\le  \( \log \frac{a}{\ep} \)^{-2} \int_{B^2_1} |\nabla (u_m -u) |^2 \,dx
+ C(\ep) \int_{B^2_1 \setminus B^2_\ep} |u_m -u|^2 \,dx,
\end{align*}
where the second inequality comes from the Hardy inequality under spherical average zero (see e.g. Proposition \ref{Prop bdd} in \S \ref{S App}). 
Letting $m \to \infty$ and $\ep \to 0$, we see that $u_m \to u$ in $L^2 \( B_1^2 ; |x|^{-2} \( \log \frac{a}{|x|} \)^{-2}\,dx \)$. 
Hence the embedding $\( H_{0, {\rm rad}}^1(B_1^2) \)^{\perp} \hookrightarrow L^2 \( B_1^2 ; |x|^{-2} (\log \frac{a}{|x|})^{-2}\,dx \)$ is compact. 
Next, we assume that $q >2$. Let 
$\{ u_m \}_{m=1}^\infty \subset \( H_{0, {\rm rad}}^1(B_1^2) \)^{\perp}$ be a bounded sequence and $u_m \rightharpoonup u$ in $\( H_{0, {\rm rad}}^1(B_1^2) \)^{\perp}$. From the case where $q=2$, we have $u_m \to u$ in $L^2 \( B_1^2 ; |x|^{-2} \( \log \frac{a}{|x|} \)^{-2}\,dx \)$. Let $1<p <\infty$. By the H\"older inequality, we have
\begin{align*}
\int_{B_1^2} \frac{|u_m -u|^q}{|x|^2 (\log \frac{a}{|x|})^{\frac{q}{2}+1}} \,dx 
\le \( \int_{B_1^2} \frac{|u_m -u|^2}{|x|^2 (\log \frac{a}{|x|})^2} \,dx \)^{\frac{1}{p}}
\( \int_{B_1^2} \frac{|u_m -u|^r}{|x|^2 (\log \frac{a}{|x|})^{\tilde{r}}} \,dx \)^{1-\frac{1}{p}},
\end{align*}
where $r= \frac{p}{p-1} \( q- \frac{2}{p} \) = q + \frac{q-2}{p-1} >2, \tilde{r} = \frac{p}{p-1} \( 1+ \frac{q}{2} -\frac{2}{p} \)$. Note that $\tilde{r} = \frac{r}{2} + 1$. By the embedding $H_0^{1} (B_1^2) \hookrightarrow L^r \( B_1^2 ; |x|^{-2} \( \log \frac{1}{|x|} \)^{-1-\frac{r}{2}}\,dx \)$, we have
\begin{align*}
\int_{B_1^2} \frac{|u_m -u|^q}{|x|^2 (\log \frac{a}{|x|})^{\frac{q}{2}+1}} \,dx 
\le \( \int_{B_1^2} \frac{|u_m -u|^2}{|x|^2 (\log \frac{a}{|x|})^2} \,dx \)^{\frac{1}{p}}
C \( \int_{B_1^2} |\nabla (u_m - u) |^2 \,dx \)^{\frac{r}{2} \( 1-\frac{1}{p} \)}
\to 0,
\end{align*}
as $m \to \infty$. 
Therefore, $\( H_{0, {\rm rad}}^1(B_1^2) \)^{\perp} \hookrightarrow L^q \( B_1^2 ; |x|^{-2} (\log \frac{a}{|x|})^{-1-\frac{q}{2}}\,dx \)$ is compact for any $q \ge 2$. 
\end{proof}

\begin{proof}{\it (Corollary \ref{Cor cpt})}
Although it is already shown by the scaling argument (see e.g.  \cite{HK,S(JDE)}), we write the proof here for reader's convenience. Since $H_{0, {\rm rad}}^{1} (B_1^2) \subset H_0^1 (B_1^2)$, it is enough to show the non-compactness of the embedding $H_{0, {\rm rad}}^{1} (B_1^2) \hookrightarrow L^2 \( B_1^2 ; |x|^{-2} \( \log \frac{a}{|x|} \)^{-2}\,dx \)$. Let $u \in C_{c, {\rm rad}}^\infty (B_1^2) \setminus \{ 0\}$. Consider $u_{\la} (x) = \la^{-\frac{1}{2}} \,u (y), \, y = \( \frac{|x|}{a} \)^{\la -1} x$ for $\la \le 1$. Note that supp $u_\la \subset B_{a^{1-\frac{1}{\la}}} \subset B^2_1$ for $\la \le 1$. Let $|x|=r, |y|=s$. Since $\frac{s}{a} =\frac{r^\la}{a^\la}$ and $\,r \, \frac{ds}{dr} = \la s$, for any $\la \le 1$ we have
\begin{align*}
\int_{B_1^2} | \nabla u_\la |^2 \,dx 
&= 2\pi \int_0^{a^{1-\frac{1}{\la}}} \left| \frac{d u_\la}{dr} (r) \right|^2 r \,dr \\
&=2\pi \la^{-1} \int_0^{1} \left| \frac{d u}{ds} (s) \right|^2 r \( \frac{ds}{dr} \) \,ds
=\int_{B_1^2} | \nabla u |^2 \,dy < \infty,\\
\int_{B_1^2} \frac{|u_\la|^q}{|x|^2 (\log \frac{a}{|x|})^{1+ q/2}} \,dx
&= 2\pi \int_0^{a^{1-\frac{1}{\la}}} \frac{|u_\la (r)|^q}{(\log \frac{a}{r})^{1+q/2}} \,\frac{dr}{r}\\
&= 2\pi \la^{-\frac{q}{2}} \int_0^{1} \frac{|u (s)|^q}{(\log \frac{a}{r})^{1+q/2}} \,\frac{ds}{\la s}\\
&= 2\pi \la^{-\frac{q}{2}} \int_0^{1} \frac{|u (s)|^q}{\( \frac{1}{\la} \log \frac{a}{s} \)^{1+q/2}} \,\frac{dr}{\la s}
=\int_{B_1^2} \frac{|u|^q}{|y|^2 (\log \frac{a}{|y|})^{1+q/2}} \,dy >0.
\end{align*}
Especially, if we choose $\la = \frac{1}{m} \,(m \in \N)$, then we see that $\{ u_{\frac{1}{m}} \}_{m=1}^\infty \subset H_{0, {\rm rad}}^1(B_1^2)$ is a non-compact sequence. In fact, we see that $u_{\frac{1}{m}} \rightharpoonup 0$ in $H_{0, {\rm rad}}^1(B_1^2)$ and $u_{\frac{1}{m}} \not\to 0$ in $L^q \( B_1^2 ; |x|^{-2} \( \log \frac{a}{|x|} \)^{-1-\frac{q}{2}}\,dx \)$. Therefore, the embedding $H_{0, {\rm rad}}^{1} (B_1^2) \hookrightarrow L^q \( B_1^2 ; |x|^{-2} \( \log \frac{a}{|x|} \)^{-1-\frac{q}{2}}\,dx \)$ is non-compact. 
\end{proof}


Next, we show Corollary \ref{Cor A_a q=2}. 

\begin{proof}{\it (Corollary \ref{Cor A_a q=2})}
Consider the following test function.
\begin{align*}
u_{\alpha} (x) = g_2 (\theta) \, f_{\alpha} (r) \,\,\( \alpha > \frac{1}{2} \),\,\text{where}\,\,
f_a (r) = 
\begin{cases}
2 (\log 2)^{\alpha} r \quad &\text{if}\,\, r \in [0, \frac{1}{2}],\\
\( \log \frac{1}{r} \)^{\alpha} \, &\text{if} \,\, r \in (\frac{1}{2}, 1),
\end{cases}
\end{align*}
and $g_2 = g_2 (\theta)$ satisfies
\begin{align*}
\int_0^{2\pi} g_2 (\theta) \,d\theta =0, \,\,\int_0^{2\pi} | {g_2}' (\theta)|^2 \,d\theta = \int_0^{2\pi} |g_2 (\theta)|^2 \,d\theta.
\end{align*}
Then we have
\begin{align*}
A_1 \le \frac{\int_{B^2_1} \left|  \nabla u_{\alpha} \right|^2 \,dx}{\int_{B_1^2} \frac{|u_\alpha|^2}{|x|^2 (\log \frac{1}{|x|})^2}\,dx}
&\le \frac{\int_{1/2}^1 \int_0^{2\pi} \left| \alpha \( \log \frac{1}{r} \)^{\alpha -1} \frac{1}{r} \, \right|^2 |g_2(\theta)|^2 r + \( \log \frac{1}{r} \)^{2\alpha} | {g_2}' (\theta)|^2 r^{-1} \,dr d\theta + C}{\int_{1/2}^{1} \int_0^{2\pi} \( \log \frac{1}{r} \)^{2\alpha -2} | g_2(\w)|^2 r^{-1} \,dr d\theta}\\
&=\alpha^2 +  \frac{\int_{1/2}^1 \( \log \frac{1}{r} \)^{2\alpha} r^{-1} \,dr  + \tilde{C}}{\int_{1/2}^{1} \( \log \frac{1}{r} \)^{2\alpha -2} r^{-1} \,dr} \\
&=\alpha^2 +  \frac{\int_0^{\log 2} t^{2\alpha}  \,dt  + \tilde{C}}{\int_0^{\log 2} t^{2\alpha -2} \,dt}
= \frac{1}{4} + o(1) \quad \( \alpha \to \frac{1}{2} \). 
\end{align*}
Since $B_1=\frac{1}{4}$ is not attained, $A_1 = \frac{1}{4}$ is not attained.

Assume that $a>1$. Let 
$\{ u_m \}_{m=1}^\infty \subset \( H_{0, {\rm rad}}^1(B_1^2) \)^{\perp}$ be a minimizing sequence of $A_a$ satisfying
\begin{align*}
\int_{B_1^2} \frac{|u_m|^2}{|x|^2 (\log \frac{a}{|x|})^2}\,dx =1, \quad \int_{B_1^2} |\nabla u_m |^2 \,dx = A_a + o(1)\,\, (m \to \infty).
\end{align*}
Since $\{ u_m \}$ is bounded, in the same way as the proof of Theorem \ref{Thm cpt}, we have $u_m \rightharpoonup u$ in $\( H_{0, {\rm rad}}^1(B_1^2) \)^{\perp}$. By Theorem \ref{Thm cpt}, we have 
\begin{align*}
&u_m \to u\,\, \text{in}\,\, L^2 \( B_1^2 ; |x|^{-2} \( \log \frac{a}{|x|} \)^{-2}\,dx \), \\
&\int_{B_1^2} \frac{|u|^2}{|x|^2 (\log \frac{a}{|x|})^2}\,dx= \lim_{m \to \infty} \int_{B_1^2} \frac{|u_m|^2}{|x|^2 (\log \frac{a}{|x|})^2}\,dx =1, \\
&\int_{B_1^2} |\nabla u |^2 \,dx \le \liminf_{m \to \infty} \int_{B_1^2} |\nabla u_m |^2 \,dx = A_a.
\end{align*}
Therefore, $u$ is a minimizer of $A_a$. Since $A_a$ is attained for $a>1$, we see that $A_a > \frac{1}{4} = B_a$. 
\end{proof}

%
%
\section{Compactness of the radial Sobolev spaces: Improvement of the compactness result in \cite{ES}}\label{S radial}

Let $N \ge 3, 2^* = \frac{2N}{N-2}, 1 \le p < \infty, L^{p,q} (\re^N)$ be the Lorentz space which is given by
\begin{align*}
L^{p,q}(\re^N) &= \left\{ \, u: \re^N \to \re\,\,{\rm measurable} \,\,\biggr| \,\, \| u\|_{p,q} < \infty  \, \right\}, \\
\| u \|_{p,q} &= 
\begin{cases}
\( \int_0^{\infty}  \( s^{\frac{1}{p}} u^* (s) \)^q \,\frac{ds}{s} \)^{\frac{1}{q}} \quad &\text{if} \,\, 1\le q <\infty, \vspace{0.5em}\\
\sup_{s \in (0, \infty)} s^{\frac{1}{p}}  u^* (s)   &\text{if} \,\, q = \infty,
\end{cases}
\end{align*}
$u^*: [0, \infty) \to [0, \infty]$ denotes the decreasing rearrangement of $u$ and $u^{\#}: \re^N \to [0, \infty]$ denotes the Schwartz symmetrization of $u$ which are given by
\begin{align*}
u^{*}(t) &= \inf \left\{ \la > 0 \,\, \Biggl| \,\, \left| \{ x \in \re^N \,\, | \,\, |u(x)| > \la \} \right| \le t \right\},\\
u^{\#} (x) =u^{\#}(|x|) &= \inf \left\{ \la > 0 \,\, \Biggl| \,\, \left| \{ x \in \re^N \,\, | \,\, |u(x)| > \la \} \right| \le |B^N_{|x|}| \right\}
\end{align*}
(Ref. \cite{BS,Lieb-Loss}). Obviously, we have
\begin{align*}
u^* ( t ) = u^{\#} (r), \,\text{where} \,\, t = |B^N_{|x|}| = \frac{\w_{N-1}}{N} r^N. 
\end{align*}

The authors \cite{ES} obtained the following compactness result in the framework of Lebesgue spaces.

\begin{ThmB}(\cite{ES} Corollary 1)
\begin{align*} 
(i) \,\, &H^{1}_{\rm rad} (\re^N) \cap L^p (\re^N ) \hookrightarrow L^2 (\re^N) \,\,\text{is compact for}\, p \in [1,2).\\
(ii) \,\, &H^{1}_{\rm rad} (\re^N) \cap L^p (\re^N ) \hookrightarrow L^{2^*} (\re^N) \,\,\text{is compact for}\, p >2^*.
\end{align*}
\end{ThmB}

\begin{remark}
Actually, we do not need the radially symmetry in (i). Namely, $H^{1}(\re^N) \cap L^p (\re^N ) \hookrightarrow L^2 (\re^N)$ is compact for $p \in [1, 2)$. 
\end{remark}

In this section, we show the following compactness result in the framework of Lorentz spaces.

\begin{theorem}\label{Thm cpt Lorentz}
Let $q \in [1, \infty]$. Then the followings hold. 
\begin{align*} 
(i) \,\, &H^{1}_{\rm rad} (\re^N) \cap L^{p, q} (\re^N ) \hookrightarrow L^2 (\re^N) \,\,\text{is compact for}\, p \in [1,2).\\
(ii) \,\, &H^{1}_{\rm rad} (\re^N) \cap L^{p, q} (\re^N ) \hookrightarrow L^{2^*} (\re^N) \,\,\text{is compact for}\, p >2^*.
\end{align*}
\end{theorem}

\begin{remark}\label{Rem L^{p,q}}
Since $L^p \subsetneq L^{p, q} \subsetneq L^{p, \infty}$ for any $q \in (p, \infty)$, Theorem \ref{Thm cpt Lorentz} is an improvement of Theorem B. 
\end{remark}

\begin{remark}
Set $u_m (x) = m^{\frac{N-2}{2}} u\( mx \)$ for $x \in \re^N, m \in \N$ and $u \in C_{c, {\rm rad}}^\infty (\re^N)$. 
Direct calculation implies that 
\begin{align*}
\| u_m \|_{p,q} = m^{\frac{N-2}{2} - \frac{N}{p}} \| u\|_{p,q} \to \infty \,\,\text{as} \,\, m\to \infty \,\, \text{if} \,\, p > 2^*.
\end{align*}
Therefore, we can 
deny the possibility of non-compactness with respect to $\dot{H}^1$ and $L^{2^*}$-invariance under the restriction of $H^1(\re^N)$ to $H^1(\re^N) \cap L^{p,q} (\re^N)$ 
when $p > 2^*$. 
\end{remark}

Let $G$ be a closed subgroup of the orthogonal group $O(N)$. 
We call $u(x)$ a $G$ {\it invariant function} if $u(gx) = u(x)$ for all $g \in G$ and $x \in \re^N$. Set 
$$H_G^1 (\re^N) = \{ \,u \in H^1 (\re^N) \,\,| \,\, u\,\,\text{is}\,\,G \,\,\text{invariant} \,\}.$$
Clearly, we see that $H^1_{\rm rad} (\re^N) = H^1_{O(N)} (\re^N)$. The following result is a generalization of Theorem A to $H^1_G (\re^N)$.
\begin{ThmC}(\cite{Li}, see also \cite{W} \S 1.5)
Let $N_j \ge 2, j=1,2, \cdots, k, \sum_{j=1}^k N_j = N$ and 
\begin{align*}
G := O(N_1) \times O(N_2) \times \cdots \times O(N_k). 
\end{align*}
Then the embedding $H^{1}_G (\re^N)  \hookrightarrow L^p (\re^N)$ is compact for $p \in (2, 2^*)$.
\end{ThmC}
In the same way, we can generalize Theorem \ref{Thm cpt Lorentz} to $H^1_G (\re^N)$.
\begin{theorem}\label{Thm cpt G}
Let $q \in [1, \infty], N_j \ge 2, j=1,2, \cdots, k, \sum_{j=1}^k N_j = N$ and 
\begin{align*}
G := O(N_1) \times O(N_2) \times \cdots \times O(N_k). 
\end{align*}Then the followings hold. 
\begin{align*} 
(i) \,\, &H^{1}_G (\re^N) \cap L^{p, q} (\re^N ) \hookrightarrow L^2 (\re^N) \,\,\text{is compact for}\, p \in [1,2).\\
(ii) \,\, &H^{1}_G (\re^N) \cap L^{p, q} (\re^N ) \hookrightarrow L^{2^*} (\re^N) \,\,\text{is compact for}\, p >2^*.
\end{align*}
\end{theorem}

Since $H^1_{\rm rad} (\re^N) \subset H^1_G (\re^N)$, Theorem \ref{Thm cpt Lorentz} follows from Theorem \ref{Thm cpt G}.
We show Theorem \ref{Thm cpt G} only in the case where $q = \infty$, see Remark \ref{Rem L^{p,q}}.

\begin{proof}{\it (Theorem \ref{Thm cpt G})}
Let $\{ v_m \}_{m=1}^\infty \subset H^{1}_G (\re^N) \cap L^{p, \infty} (\re^N )$ be a bounded sequence and $v_m \rightharpoonup v$ in $H^{1}_G (\re^N) \cap L^{p, \infty} (\re^N )$. 
Set $u_m = v_m - v$. Then $u_m \rightharpoonup 0$ in $H^{1}_G (\re^N)$. From Theorem C, we have $u_m \to 0$ in $L^q (\re^N)$ for $q \in (2, 2^*)$.
By the P\'{o}lya-Szeg\"{o} inequality, we have
\begin{align*}
\| u^{\#}_m \|^2_{H^1 (\re^N)} = \| \nabla u^{\#}_m \|_{L^2 (\re^N)}^2 + \| u_m^{\#} \|_{L^2 (\re^N)}^2
\le \| \nabla u_m \|_{L^2 (\re^N)}^2 + \| u_m \|_{L^2 (\re^N)}^2 
= \| u_m \|^2_{H^1 (\re^N)} < \infty
\end{align*}
which implies that $\{ u_m^{\#} \}_{m=1}^\infty$ be the bounded sequence in $H^{1}_{\rm rad} (\re^N)$. Therefore, there exists $u \in H^{1}_{\rm rad} (\re^N)$ such that $u_m^{\#} \rightharpoonup u$ in $H^{1}_{\rm rad} (\re^N)$. From Theorem A,  $u_m^{\#} \to u$ in $L^q (\re^N)$ for $q \in (2, 2^*)$. Since 
\begin{align*}
0 = \lim_{m \to \infty} \| u_m \|_q = \lim_{m \to \infty} \| u_m^{\#} \|_q = \| u \|_q,
\end{align*}
we have $u \equiv 0$. Therefore, $u_m^{\#} \rightharpoonup 0$ in $H^{1}_{\rm rad} (\re^N)$.

\noindent
(i) Let $p<2$. Since $H^1 (B_R^N) \hookrightarrow L^2 (B^N_R)$ is compact, we have $u_m^{\#} \to 0$ in $L^2(B_R^N)$. Then, we have
\begin{align*}
\int_{B_R^N} |u_m^{\#} |^2 \,dx &=: C_1 (m, R) \to 0 \quad \text{as} \,\, m \to \infty\,\,\text{for fixed}\,\, R>0.
\end{align*}
On the other hand, 
\begin{align*}
\int_{\re^N} |u_m |^2 \,dx 
&=\int_{\re^N} |u_m^{\#} |^2 \,dx \\
&=\int_{\re^N \setminus B_R^N} |u_m^{\#} |^2 \,dx  + C_1 (m, R) \\
&=\w_{N-1} \int_R^\infty |u_m^{\#} (r) |^2 r^{N-1} \,dr  + C_1 (m, R) \\
&=\int_R^\infty |u_m^* (t) |^2 \,dt  + C_1 (m, R) \\
&\le \| u_m \|^2_{p, \infty} \int_R^\infty t^{-\frac{2}{p}} \,dt + C_1 (m, R) 
= C R^{-\frac{2}{p} +1} + C_1 (m, R).
\end{align*}
Since $p < 2$, 
\begin{align*}
\lim_{R \to \infty} \lim_{m \to \infty} \int_{\re^N} |u_m |^{2} \,dx 
= 0
\end{align*}
which implies that $v_m \to v$ in $L^2 (\re^N)$. 
Hence, the embedding $H^{1}_G (\re^N) \cap L^{p, \infty} (\re^N ) \hookrightarrow L^2 (\re^N)$ is compact for $p \in [1, 2)$.

\noindent
(ii) Let $p>2^*$. First, we recall the following pointwise estimate for any radial function $f$.
\begin{align}\label{Radial lemma}
| f(x) |\le \sqrt{\frac{2}{\w_{N-1}} } \,\,\| f \|^{\frac{1}{2}}_{2} \,\| \nabla f \|^{\frac{1}{2}}_2 \, |x|^{-\frac{N-1}{2}} \,\,\text{a.e. in} \,\, x \in \re^N
\end{align}
See e.g. \cite{HS} Proposition 6. 
By (\ref{Radial lemma}), we have
\begin{align}\label{soto1}
\int_{\re^N \setminus B^N_R} |u_m^{\#}|^{2^*} \,dx 
&\le C \, \| u_m^{\#} \|_{H^1 (\re^N)}^{2^*} \int_{\re^N \setminus B^N_R} |x|^{-\frac{N(N-1)}{N-2}} \,dx \le C R^{-\frac{N}{N-2}}. 
\end{align}
By the compactness of the one-dimensional Sobolev embedding $H^1 (\ep, R)  \hookrightarrow L^s (\ep, R)$ for any $s \ge 1$ and the equivalence of $H^1_{{\rm rad}} (B^N_R \setminus B^N_\ep)$ and $H^1 ((\ep, R) ; r^{N-1} \,dr) \simeq H^1 (\ep, R)$, we have the compactness of the embedding $H^1_{\rm rad} (B^N_R \setminus B^N_\ep) \hookrightarrow L^s (B^N_R \setminus B^N_\ep)$ for any $s \ge 1$ and for any fixed $R > \ep >0$. Therefore, we have
\begin{align}\label{aida1}
\int_{B^N_R \setminus B^N_\ep} |u_m^{\#}|^{2^*} \,dx =: C_2 (m, \ep, R) \to 0 \quad \text{as} \,\, m \to \infty\,\,\text{for fixed}\,\, R, \ep >0. 
\end{align}
By (\ref{soto1}) and (\ref{aida1}), we have 
\begin{align*}
\int_{\re^N} |u_m |^{2^*} \,dx 
&= \int_{\re^N} |u_m^{\#} |^{2^*} \,dx \\
&\le \int_{B_\ep^N} |u_m^{\#} |^{2^*} \,dx  + C_2 (m, \ep, R) + C R^{-\frac{N}{N-2}} \\
&=\int_0^\ep |u_m^* (t) |^{2^*} \,dt  + C_2 (m, \ep, R) + C R^{-\frac{N}{N-2}} \\
&\le \| u_m \|^{2^*}_{p, \infty} \int_0^\ep t^{-\frac{2^*}{p}} \,dt + C_2 (m, \ep, R) + C R^{-\frac{N}{N-2}} \\
&= C \ep^{-\frac{2^*}{p} +1} + C_2 (m, \ep, R) + C R^{-\frac{N}{N-2}}.
\end{align*}
Since $p > 2^*$, 
\begin{align*}
\lim_{\ep \to 0} \lim_{R \to \infty} \lim_{m \to \infty} \int_{\re^N} |u_m |^{2^*} \,dx 
= 0
\end{align*}
which implies that $v_m \to v$ in $L^{2^*} (\re^N)$. 
Hence, the embedding $H_G^{1} (\re^N) \,\cap \, L^{p, \infty}(\re^N) \hookrightarrow L^{2^*} (\re^N)$ is compact for $p > 2^*$.  
\end{proof}

In the end of this section, we give remarks about H\"older inequalities and interpolation inequalities, and give another proof of Theorem \ref{Thm cpt G}. 
Theorem B was shown in \cite{ES} by using Theorem A and the H\"older inequality: $\| f g \|_r \le \| f \|_q \| g \|_p$, where $\frac{1}{r} = \frac{1}{p} + \frac{1}{q}$. Especially, they used the following interpolation inequalities by setting $f=|u|^\la, g=|u|^{1-\la}$. 
\begin{align*}
&{\rm (i)} \,\,\| u \|_2 \le \| u\|_q^\la \, \|u \|_p^{1-\la} \quad \( \,p < 2 < q < 2^*, \la = \frac{q (2-p)}{2 (q-p)} \, \)\\
&{\rm (ii)} \,\,\| u \|_{2^*} \le \| u\|_q^\la \, \|u \|_p^{1-\la} \quad \( \,2 < q < 2^* < p, \la = \frac{q (2^* -p)}{2^* (q-p)} \,  \)
\end{align*}
In fact, if we assume that $u_m \rightharpoonup u$ in $H^{1}_{\rm rad} (\re^N) \cap L^p (\re^N )$, then we have $u_m \to u$ in $L^q (\re^N)$ from Theorem A. By using the interpolation inequalities above, we obtain $u_m \to u$ in $L^2$ in the case (i) and  $u_m \to u$ in $L^{2^*} (\re^N)$ in the case (ii). Therefore, (i) and (ii) in Theorem B hold. 
However, in the weak-Lebesgue spaces $L^{p, \infty}(\re^N)$, the H\"older  inequality does not hold in general. In fact, the following will be proved
\begin{align}\label{Holder muri}
{\not\exists}C >0\,\, \text{s.t.}\,\,\| f g \|_2 \le C \| f\|_q \, \|g \|_{p, \infty} \,\, \({\forall}f \in L^q (\re^N), \, {\forall}g \in L^{p, \infty} (\re^N) \),
\end{align}
where $p < 2 < q < 2^*$. Indeed, if we consider 
\begin{align*}
f(x) = |x|^{-\alpha} I_{B_1(0)} (x), \, g(x) = |x|^{-\frac{N}{p}}, \, \alpha = \frac{N}{q} -\ep,\, \ep >0,
\end{align*}
then
\begin{align*}
&\| f \|_{q} =  \( \int_{B_1 (0)} |x|^{-q \alpha}\,dx \)^{\frac{1}{q}} 
=\( \int_{B_1 (0)} |x|^{-N+ q \ep}\,dx \)^{\frac{1}{q}} = C \ep^{-\frac{1}{q}}, \\
&\| g \|_{p, \infty} < C,\\
&\| f g \|_{2} =  \( \int_{B_1 (0)} |x|^{-2\alpha -\frac{2N}{p}}\,dx \)^{\frac{1}{2}} 
=  \( \int_{B_1 (0)} |x|^{-N + 2\ep}\,dx \)^{\frac{1}{2}} = C \ep^{-\frac{1}{2}}.
\end{align*}
Therefore, 
\begin{align*}
C \ep^{-\frac{1}{2}} = \| f g \|_{2}  \le C \| f \|_{q} \| g \|_{p, \infty} = C \ep^{-\frac{1}{q}} 
\end{align*}
which implies (\ref{Holder muri}) when we take $\ep \to 0$. 
However, in spite of (\ref{Holder muri}), we can show the following interpolation inequalities from Proposition \ref{Prop Holder} below.
\begin{align*}
&{\rm (i)} \,\,\| u \|_2 \le C \,\| u\|_q^\la \, \|u \|_{p, \infty}^{1-\la} \quad \( \,p < 2 < q < 2^*, \la \in (0,1) \, \)\\
&{\rm (ii)} \,\,\| u \|_{2^*} \le C \,\| u\|_q^\la \, \|u \|_{p, \infty}^{1-\la} \quad \( \,2 < q < 2^* < p, \la \in (0,1) \,  \)
\end{align*}
As a consequence, thanks to Proposition \ref{Prop Holder}, we can show Theorem \ref{Thm cpt G} in the same way as the proof of Theorem B.

\begin{prop}\label{Prop Holder}
Let $1 \le p < q < r \le \infty$. Then the interpolation inequality 
\begin{align*}
\,\,\| u \|_q \le D\,\| u\|_{p, \infty}^\la \, \|u \|_{r, \infty}^{1-\la}, 
\end{align*}
holds for any $u \in L^{p, \infty} (\re^N) \cap L^{r, \infty} (\re^N)$, where $D= \( \frac{q(r-p )}{(r-q)(q-p)} \)^{\frac{1}{q}}$ and $\la = \frac{p(r-q)}{q (r-p)}$. 
\end{prop}

\begin{proof}
For any $s >0$, we have
\begin{align*}
\| u \|_q^q &= \int_0^s |u^* (t)|^q \,dt + \int_s^\infty |u^* (t)|^q \,dt\\
&\le \| u \|_{r, \infty}^q \int_0^s t^{-\frac{q}{r}} \,dt + \| u \|_{p, \infty}^q \int_s^\infty t^{-\frac{q}{p}} \,dt 
=  A s^a + B s^{-b},
\end{align*}
where
\begin{align*}
A= \frac{r}{r-q} \| u \|_{r, \infty}^q,\,\, a= \frac{r-q}{r}, \,\, B= \frac{p}{q-p} \| u \|_{p, \infty}^q,\,\, b= \frac{q-p}{p}.
\end{align*} 
Note that the function $f(s) =  A s^a + B s^{-b}$ attains its minimum at $s= \( \frac{Bb}{Aa} \)^{\frac{1}{a+b}}$. Since $ \( \frac{Bb}{Aa} \)^{\frac{1}{a+b}} = \( \frac{\| u \|_{p, \infty}}{\| u \|_{r, \infty}} \)^{\frac{pr}{r-p}}$, we have
\begin{align*}
\| u \|_q^q 
\le \frac{r}{r-q} \| u \|_{r, \infty}^q \( \frac{\| u \|_{p, \infty}}{\| u \|_{r, \infty}} \)^{\frac{p(r-q)}{r-p}} + \frac{p}{q-p} \| u \|_{p, \infty}^q \( \frac{\| u \|_{p, \infty}}{\| u \|_{r, \infty}} \)^{-\frac{r(q-p)}{r-p}}
= D^q\,\| u\|_{p, \infty}^{q\la} \, \|u \|_{r, \infty}^{q(1-\la )}. 
\end{align*}
\end{proof}

%
%
\section{Appendix}\label{S App}

First, we give a proof of two Hardy type inequalities under spherical average zero on the ball. 
 
\begin{prop}\label{Prop bdd}
For any $u \in C_0^\infty (B_1^N)$ with $\int_{\mathbb{S}^{N-1}} u(r\w) \,dS_\w =0 \,(\forall r \ge 0)$, the following inequalities hold.
\begin{align}\label{H_2}
&\frac{N^2}{4} \int_{B_1^N} \frac{|u|^2}{|x|^2}\,dx \le \int_{B^N_1} \left|  \nabla u  \right|^2 \,dx\\
\label{CH_2}
&\frac{5}{4} \int_{B^2_1} \frac{|u|^2}{|x|^2 (\log \frac{e}{|x|})^2}\,dx \le \int_{B^2_1} \left|  \nabla u  \right|^2 \,dx
\end{align}
Moreover, the best constant $\frac{N^2}{4}$ in (\ref{H_2}) is not attained. 
\end{prop}

\begin{proof}
We use the polar coordinate:
\begin{align*}
x=r\w \,(r=|x|, \w \in \Sh^{N-1}),\quad \nabla u(x) = \( \frac{\pd u}{\pd r}(r\w) \) \w + \frac{1}{r^2} \nabla_{\Sh^{N-1}} u(r \w)
\end{align*}
From the classical Hardy inequality:
\begin{align}\label{standard H_2}
\( \frac{N-2}{2} \)^2 \int_{B_1^N} \frac{|u|^2}{|x|^2}\,dx < \int_{B^N_1} \left|  \nabla u \cdot \frac{x}{|x|} \right|^2 \,dx \quad ({\forall}u \in C_0^\infty (B_1^N))
\end{align}
and the Poincar\'e inequality on the sphere $\Sh^{N-1}$:
\begin{align}\label{Poincare}
(N-1) \int_{\Sh^{N-1}} |g(\w)|^2 \,dS_\w \le \int_{B^N_1} \left|  \nabla_{\Sh^{N-1}} g \right|^2 \,dx \,\,({\forall}g \in C^\infty (\Sh^{N-1}), \,\int_{\Sh^{N-1}} g \,dS_\w =0 ),
\end{align}
for any $u \in C_0^\infty (B_1^N)$ with $\int_{\mathbb{S}^{N-1}} u(r\w) \,dS_\w =0 \,(\forall r \ge 0)$ we have
\begin{align*}
\int_{B^N_1} \left|  \nabla u  \right|^2 \,dx
&= \int_{B^N_1} \left|  \nabla u \cdot \frac{x}{|x|} \right|^2 \,dx 
+ \int_0^1 \int_{\Sh^{N-1}} \left|  \nabla_{\Sh^{N-1}} u(r\w)  \right|^2 r^{N-3}\,dr dS_\w \\
&> \( \frac{N-2}{2} \)^2 \int_{B_1^N} \frac{|u|^2}{|x|^2}\,dx + (N-1) \int_0^1 \int_{\Sh^{N-1}} \left|  u(r\w)  \right|^2 r^{N-3}\,dr dS_\w \\
&= \frac{N^2}{4} \int_{B_1^N} \frac{|u|^2}{|x|^2}\,dx
\end{align*}
which implies the inequality (\ref{H_2}). It is enough to show the optimality of the constant $\frac{N^2}{4}$ in (\ref{H_2}). Consider the following test function.
\begin{align*}
u_a(x) = g_2 (\w) \, f_a (r) \,\,\( a > 0 \),\,\,\text{where}\,\,
f_a (r) = 
\begin{cases}
r^a \quad &\text{if}\,\, r \in [0, \frac{1}{2}],\\
\text{smooth} \, &\text{if} \,\, r \in (\frac{1}{2}, 1),\\
0 \, &\text{if} \,\, r \in [1, \infty)
\end{cases}
\end{align*}
and $g_2$ is the second eigenfunction of the Laplace-Beltrami operator $-\lap_{\Sh^{N-1}}$ which satisfies 
\begin{align*}
\int_{\Sh^{N-1}} g_2 (\w) \,dS_\w =0, \,\,\int_{\Sh^{N-1}} |\nabla_{\Sh^{N-1}} g_2 (\w)|^2 \,dS_\w = (N-1)\int_{\Sh^{N-1}} |g_2 (\w)|^2 \,dS_\w.
\end{align*}
Then, we have
\begin{align*}
\frac{\int_{B^N_1} \left|  \nabla u_a  \right|^2 \,dx}{\int_{B_1^N} \frac{|u_a|^2}{|x|^2}\,dx}
&\le \frac{\int_0^{1/2} \int_{\Sh^{N-1}} |a r^{a-1}|^2 |g_2(\w)|^2 r^{N-1} + r^{2a} |\nabla_{\Sh^{N-1}} g_2(\w)|^2 r^{N-3} \,dr dS_\w + C}{\int_0^{1/2} \int_{\Sh^{N-1}} r^{2a} | g_2(\w)|^2 r^{N-3} \,dr dS_\w}\\
&= \( a^2 + N-1 \) + \frac{C}{\int_0^{1/2} \int_{\Sh^{N-1}} r^{2a} | g_2(\w)|^2 r^{N-3} \,dr dS_\w}\\
&= \( a^2 + N-1 \) + R(a, N). 
\end{align*}
Since $R(a,N) \to 0$ as $a \to 0$, we can obtain the optimality of the constant $1=\frac{N^2}{4}$ when $N=2$. 
In the case where $N \ge 3$, consider the following test function.
\begin{align*}
v_m(x) = g_2 (\w) \, h_m (r),\,\,\text{where}\,\,
h_m (r) = 
\begin{cases}
0 \, \quad 
&\text{if} \,\, r \in [0, \frac{1}{2m}],\\
2 m \, (m^{\frac{N-2}{2}} -1) \( r - \frac{1}{2m} \) 
&\text{if}\,\, r \in (\frac{1}{2m}, \frac{1}{m}),\\
r^{-\frac{N-2}{2}} -1 \, &\text{if} \,\, r \in [\frac{1}{m}, 1].
\end{cases}
\end{align*}
Then we have
\begin{align*}
&\frac{\int_{B^N_1} \left|  \nabla v_m  \right|^2 \,dx}{\int_{B_1^N} \frac{|v_m|^2}{|x|^2}\,dx} \\
&\le \frac{\int_{1/m}^1 \int_{\Sh^{N-1}} \left| -\frac{N-2}{2} r^{-\frac{N}{2}} \right|^2 |g_2(\w)|^2 r^{N-1} +  \left| r^{-\frac{N-2}{2}} -1 \right|^2 |\nabla_{\Sh^{N-1}} g_2(\w)|^2 r^{N-3} \,dr dS_\w + C_1}{\int_{1/m}^{1} \int_{\Sh^{N-1}} \left| r^{-\frac{N-2}{2}} -1 \right|^2 | g_2(\w)|^2 r^{N-3} \,dr dS_\w}\\
&=\frac{\left[ \( \frac{N-2}{2} \)^2 + N-1 \right] \log m + C_1}{\log m + C_2}
= \frac{N^2}{4} + o(1) \quad (m \to \infty)
\end{align*}
which implies the optimality of the constant $\frac{N^2}{4}$ when $N \ge 3$. 

On the other hand, if we use the critical Hardy inequality:
\begin{align}\label{standard CH_2}
\frac{1}{4} \int_{B_1^2} \frac{|u|^2}{|x|^2 \( \log \frac{e}{|x|} \)^2}\,dx < \int_{B^2_1} \left|  \nabla u \cdot \frac{x}{|x|} \right|^2 \,dx \quad ({\forall}u \in C_0^\infty (B_1^2))
\end{align}
instead of the classical Hardy inequality, for $u \in C_0^\infty (B_1^N)$ with $\int_{\mathbb{S}^{N-1}} u(r\w) \,dS_\w =0 \,(\forall r \ge 0)$ we have
\begin{align*}
\int_{B^2_1} \left|  \nabla u  \right|^2 \,dx
&= \int_{B^2_1} \left|  \nabla u \cdot \frac{x}{|x|} \right|^2 \,dx 
+ \int_0^1 \int_{\Sh^{1}} \left|  \nabla_{\Sh^{1}} u(r\w)  \right|^2 r^{-1}\,dr dS_\w \\
&> \frac{1}{4} \int_{B_1^2} \frac{|u|^2}{|x|^2 \( \log \frac{e}{|x|} \)^2}\,dx + \int_0^1 \int_{\Sh^{1}} \left|  u(r\w)  \right|^2 r^{-1}\,dr dS_\w \\
&=  \frac{1}{4} \int_{B_1^2} \frac{|u|^2}{|x|^2 \( \log \frac{e}{|x|} \)^2}\,dx + \int_{B_1^2} \frac{|u|^2}{|x|^2}\,dx \ge \frac{5}{4} \int_{B_1^2} \frac{|u|^2}{|x|^2 \( \log \frac{e}{|x|} \)^2}\,dx
\end{align*}
which implies the inequality (\ref{CH_2}).
\end{proof}

Next, we give a simple proof of the Hardy inequality under spherical average zero on the whole space. For another proof, see e.g. \cite{BMO}.  

\begin{prop}\label{Prop R^N}
For any $u \in C_0^\infty (\re^N)$ with $\int_{\mathbb{S}^{N-1}} u(r\w) \,dS_\w =0 \,(\forall r \ge 0)$, the inequality  
\begin{align*}
&\frac{N^2}{4} \int_{\re^N} \frac{|u|^2}{|x|^2}\,dx \le \int_{\re^N} \left|  \nabla u  \right|^2 \,dx
\end{align*}
holds. Moreover, the best constant $\frac{N^2}{4}$ is not attained.
\end{prop}

\begin{proof}
We show only the optimality of the constant $\frac{N^2}{4}$. Consider the following test function.
\begin{align*}
u_a(x) = g_2 (\w) f_a (r) \,\,\( a > \frac{N-2}{2} \),\,\text{where}\,\,
f_a (r) = 
\begin{cases}
r \quad &\text{if}\,\, r \in [0, 1],\\
r^{-a} \, &\text{if} \,\, r \in (1, \infty)
\end{cases}
\end{align*}
and $g_2$ is given by the proof of Proposition \ref{Prop bdd}. Then we have
\begin{align*}
\frac{\int_{\re^N} \left|  \nabla u_a  \right|^2 \,dx}{\int_{\re^N} \frac{|u_a|^2}{|x|^2}\,dx}
&\le \frac{\int_0^{1} r^{N-1}\,dr + \int_1^\infty |-a r^{-a-1}|^2 r^{N-1} + (N-1) \left[ \int_0^1 r^{N-1} + \int_1^\infty r^{-2a +N-3} \,dr \right] }{\int_1^\infty r^{-2a +N-3} \,dr}\\
&= \( a^2 + N-1 \) + \frac{1}{\int_1^\infty r^{-2a +N-3} \,dr}\\
&= \( a^2 + N-1 \) + R(a, N). 
\end{align*}
Since $R(a, N) \to 0$ as $a \to \frac{N-2}{2}$, we see that the constant $\frac{N^2}{4}$ is optimal. 
\end{proof}


Finally, we show the following result to show Theorem A in Introduction. 

\begin{prop}\label{Prop non-cpt}
Let $X$ be a Hilbert space, $Y$ be a Banach space, $A \subset X$ be a closed subspace and $X=A \oplus A^{\perp}$. If the embedding $X \hookrightarrow Y$ is non-compact and $A \hookrightarrow Y$ is compact, then $A^\perp \hookrightarrow Y$ is non-compact. 
\end{prop}

\begin{proof}
Assume that $A^{\perp} \hookrightarrow Y$ is compact. For any bounded sequence $\{ u_m \}_{m=1}^\infty \subset X$, there exist $\{ v_m \}_{m=1}^\infty \subset A$ and $\{ w_m \}_{m=1}^\infty \subset A^\perp$ such that $u_m = v_m + w_m,\, (v_m, w_m)_X  =0$ for any $m \in \N$. Since $X$ is reflexive and $X \hookrightarrow Y$, there exist $u \in X$ and a subsequence $\{ u_{m_j} \}_{j=1}^\infty$ (we use $\{ u_m \}$ again) such that $u_{m} \rightharpoonup u$ in $X, Y$. Moreover, since both sequences $\{ v_m \}_{m=1}^\infty \subset A, \{ w_m \}_{m=1}^\infty \subset A^\perp$ are also bounded and any closed subspace in reflexive Banach space is also reflexive (see e.g. \cite{B} Proposition 3.20), we have $v_m \rightharpoonup v$ in $A$ and $w_m \rightharpoonup w$ in $A^\perp$. Since both $A \hookrightarrow Y$ and $A^\perp \hookrightarrow Y$ are compact, we have $v_m \to v$ and $w_m \to w$ in $Y$. Therefore, we see that $u_m \to v+w$ in $Y$. Thanks to the uniqueness of the weak limit, we have $u= v+ w$ which implies that $u_m \to u$ in $Y$. This contradicts the non-compactness of the embedding $X \hookrightarrow Y$. Hence, the embedding $A^{\perp} \hookrightarrow Y$ is non-compact. 
\end{proof}


\section*{Acknowledgment}
The first author (S.M.) was supported by JSPS KAKENHI Grant-in-Aid for Scientific Research C 16K05191.
The second author (M.S.) was supported by JSPS KAKENHI Early-Career Scientists, No. JP19K14568. 
This work was partly supported by Osaka City University Advanced
Mathematical Institute (MEXT Joint Usage/Research Center on Mathematics
and Theoretical Physics).



\begin{thebibliography}{99}
\bibitem{AS}
Adimurthi, Sandeep, K., {\it Existence and non-existence of the first eigenvalue of the perturbed Hardy-Sobolev operator}, \newblock Proc. Roy. Soc. Edinburgh Sect. A \textbf{132} (2002), No.5, 1021-1043.



\bibitem{BS}
Bennett, C., Sharpley, R., {\it Interpolation of Operators}, 
Pure and Applied Mathematics, vol. 129, Boston Academic Press, Inc., (1988).

\bibitem{BMO}
Bez, N., Machihara, S., Ozawa, {\it Hardy type inequalities with spherical derivatives}, SN Partial Differ. Equ. Appl. (2020) 1:5. 

\bibitem{BL}
Birman, MSh, Laptev, A., {\it The negative discrete spectrum of a two-dimensional Schr\"odinger operator}, 
Commun. Pure Appl. Math. 49, 967-997 (1996). 

\bibitem{B}
Brezis, H., {\it Functional analysis, Sobolev spaces and partial differential equations}, 
Universitext. Springer, New York, 2011.


\bibitem{CM}
Caldiroli, P., Musina, R., {\it Rellich inequalities with weights}, Calc. Var. Partial Differential Equations 45 (2012), no. 1-2, 147-164.

\bibitem{CPR}
Chabrowski, J., Peral, I., Ruf, B., {\it On an eigenvalue problem involving the Hardy potential}, Commun. Contemp. Math. 12 (2010), no. 6, 953-975.

\bibitem{ES}
Ebihara, Y., Schonbek, T. P., 
{\it On the (non)compactness of the radial Sobolev spaces}, 
Hiroshima Math. J. 16 (1986), no. 3, 665-669.

\bibitem{EF}
Ekholm, T., Frank, R.L., {\it On Lieb-Thirring inequalities for Schr\"odinger operators with virtual level}, 
Commun. Math. Phys. 264, 725-740 (2006). 


\bibitem{FZZ}
Fan, X., Zhao, Y., Zhao, D.,
{\it Compact imbedding theorems with symmetry of Strauss-Lions type for the space 
$W^{1,p(x)}(\Omega)$. (English summary)}, 
J. Math. Anal. Appl. 255 (2001), no. 1, 333-348. 
g
\bibitem{HS}
Hashizume, M., Sano, M., {\it Strauss's radial compactness and nonlinear elliptic equation involving a variable critical exponent}, J. Funct. Spaces 2018, Art. ID 5497172, 13 pp.

\bibitem{H}
Hersch, J., {\it Transplantation harmonique, transplantation par modules, et th\'eor\`emes isop\'erim\'etriques. (French. English summary)}, 
Comment. Math. Helv. 44 (1969), 354-366.

\bibitem{HK}
Horiuchi, T., Kumlin, P., {\it On the Caffarelli-Kohn-Nirenberg-type inequalities involving critical and supercritical weights}, Kyoto J. Math. 52 (2012), no. 4, 661-742. 

\bibitem{II}
Ioku, N., Ishiwata, M., {\it A Scale Invariant Form of a Critical Hardy Inequality}, Int. Math. Res. Not. IMRN (2015), no. 18, 8830-8846.



\bibitem{Lieb-Loss}
Lieb, E., and Loss, M., 
{\it Analysis (second edition)}, 
\newblock Graduate Studies in Mathematics, 14, Amer. Math. Soc. Providence, RI, (2001), xxii+346 pp. 

\bibitem{Li}
Lions, P.-L.,
{\it Sym\'etrie et compacit\'e dans les espaces de Sobolev. (French. English summary) [Symmetry and compactness in Sobolev spaces]},
J. Funct. Anal. 49 (1982), no. 3, 315-334.



\bibitem{S(JDE)}
Sano, M., {\it Extremal functions of generalized critical Hardy inequalities}, J. Differential Equations 267 (2019), no. 4, 2594-2615.


\bibitem{S(NA)}
Sano, M., {\it Minimization problem associated with an improved Hardy-Sobolev type inequality}, Nonlinear Anal. 200 (2020), 111965, 16 pp.

\bibitem{ST}
Sano, M., Takahashi, F., {\it On eigenvalue problems involving the critical Hardy potential and Sobolev type inequalities with logarithmic weights in two dimensions}, in preparation.


\bibitem{ST(HT)}
Sano, M., Takahashi, F., {\it The critical Hardy inequality on the half-space via harmonic transplantation}, Calc. Var. Partial Differential Equations 61 (2022), no. 4, Paper No. 158.



\bibitem{St}
Strauss, W. A., {\it Existence of solitary waves in higher dimensions}, Comm. Math. Phys. \textbf{55} (1977), no. 2, 149-162.





\bibitem{W}
Willem, M., {\it Minimax theorems}, Progress in Nonlinear Differential Equations and their Applications, 24. Birkh\"auser Boston, Inc., Boston, MA, 1996. x+162 pp.

\bibitem{Y}
Yafaev, D., {\it Sharp constants in the Hardy-Rellich inequalities}, J. Funct. Anal. 168, 121-144 (1999).
\end{thebibliography}
\end{document}